\documentclass{amsart}

\makeatletter
\@namedef{subjclassname@2020}{%
  \textup{2020} Mathematics Subject Classification}
 
\makeatother 

\usepackage{amsthm,amssymb,amsfonts,latexsym,mathtools,thmtools,amsmath,lscape}
\usepackage[T1]{fontenc}
\usepackage{tikz-cd} 
\usepackage{enumitem} 
\usepackage{longtable}
\usepackage{multirow}
\usepackage{mathrsfs} 
\usepackage{hyperref} 
\usepackage{hyperref}
\usepackage[all]{xy}
\usepackage{booktabs}
\hypersetup{
    colorlinks=true,
    linkcolor=blue,
    filecolor=blue,      
    urlcolor=blue,
    linktocpage=true
}

\newtheorem{theorem}{Theorem}[section]

\newtheorem{proposition}[theorem]{Proposition}

\theoremstyle{definition}
\newtheorem{definition}[theorem]{Definition}

\newtheorem{remark}[theorem]{Remark}

\newtheorem{example}[theorem]{Example}

\theoremstyle{remark}

\numberwithin{equation}{section}

\begin{document}

\title[Smooth geometry of bi-quadratic algebras on three generators]{Smooth geometry of bi-quadratic algebras \\ on three generators with PBW basis}


\author{Andr\'es Rubiano}
\address{Universidad Nacional de Colombia - Sede Bogot\'a}
\curraddr{Campus Universitario}
\email{arubianos@unal.edu.co}
\thanks{}


\author{Armando Reyes}
\address{Universidad Nacional de Colombia - Sede Bogot\'a}
\curraddr{Campus Universitario}
\email{mareyesv@unal.edu.co}

\thanks{This work was supported by Faculty of Science, Universidad Nacional de Colombia - Sede Bogot\'a, Colombia [grant number 53880].}


\date{}

\dedicatory{Dedicated to Professor Oswaldo Lezama on the Occasion of His 68th Birthday}

\begin{abstract} 

In this paper, we investigate the differential smoothness of bi-quadratic algebras on three generators with PBW basis.

\end{abstract}

\maketitle


\section{Introduction}

The {\em theory of connections} in noncommutative geometry is well-known (for more details, see the beautiful treatments presented by Connes \cite{Connes1994} or Giachetta et al. \cite{Giachettaetal2005}). Briefly, one considers a differential graded algebra $\Omega A = \bigoplus\limits_{n = 0} \Omega^{n} A$ over a $\Bbbk$-algebra $A = \Omega^{0} A$ with $\Bbbk$ a field, and then defines a {\em connection} in a left $A$-module $M$ as a linear map $\nabla^{0} : M \to \Omega^{1} A \otimes_A M$ that satisfies the Leibniz rule $\nabla^0(am) = da \otimes_A m + a\nabla^{0}(m)$ for all $m\in M$ and $a \in A$. As it can be seen, this is a noncommutative definition obtained by a replacement of commutative algebras of functions on a manifold $X$, and their modules of sections of a vector bundle over $X$ (in the classical definition of a connection), by noncommutative algebras and their general one-sided modules. Just as Brzezi{\'n}ski said, \textquotedblleft this captures very well the classical context in which connections appear and brings it successfully to the realm of noncommutative geometry\textquotedblright\ \cite[p. 557]{Brzezinski2008}.

Brzezi{\'n}ski in his paper noted that, on the algebraic side, this definition of connection seems to be only a half of a more general picture. In the first place, a noncommutative connection is defined by using the tensor functor, and as is well-known, this functor has a right adjoint, the {\em hom-functor}, so it is natural to ask whether it is possible to introduce connection-like objects defined with the use of the hom-functor. In the second place, the vector space dual to $M$ is a right $A$-module and a left connection in the above sense does not induce a right connection on the dual of $M$, so having in mind the adjointness properties between tensor and hom functors, the induced map necessarily involves the hom-functor. 

Motivated by all these facts, Brzezi{\'n}ski \cite{Brzezinski2008} showed that there is a natural and potentially rich theory of connnection-like objects defined as maps on the spaces of morphisms of modules. Due to the role of spaces of homomorphisms, these objects are termed {\em hom-connections} (also are called {\em divergences} due to that if $A$ is an algebra of functions on the Euclidean space $\mathbb{R}^n$ and $\Omega^{1}(A)$ is the standard module of one-form, then we obtain the classical divergence of the elementary vector calculus \cite[p. 892]{Brzezinski2011}). As a matter of fact, he proved that hom-connections arise naturally from (strong) connections in {\em noncommutative principal bundles}, and that every left connection on a bimodule (in the sense of Cuntz and Quillen \cite{CuntzQuillen1995}) gives rise to a hom-connection. 
Brzezi{\'n}ski also studied the induction procedure of hom-connections via differentiable bimodules (and hence, via maps of differential graded algebras), and proved that any hom-connection can be extended to {\em higher forms}. He introduced the notion of {\em curvature} and showed that a consecutive application of hom-connections can be expressed in terms of the curvature, which leads to a chain complex associated to a {\em flat} (i.e. curvature-zero) {\em hom-connection} (this chain complex and its homology can be considered as dual complements of the cochain complex associated to a connection and the twisted cohomology, which is crucial in the theory of noncommutative differential fibrations \cite{BeggsBrzezinski2005}).

Two years later, Brzezi{\'n}ski et al. \cite{BrzezinskiElKaoutitLomp2010} presented a construction of {\em differential calculi} which admits hom-connections. This construction is based on the use of {\em twisted multi-derivations}, where the constructed first-order calculus $\Omega^{1}(A)$ is free as a left and right $A$-module; $\Omega^{1}(A)$ should be understood as a module of sections on the cotangent bundle over a manifold represented by $A$, and hence their construction corresponds to parallelizable manifolds or to an algebra of functions on a local chart. One year later, Brzezi{\'n}ski asserted that \textquotedblleft one should expect $\Omega^1(A)$ to be a finitely generated and projective module over $A$ (thus corresponding to sections of a non-trivial vector bundle by the Serre-Swan theorem)\textquotedblright\ \cite[p. 885]{Brzezinski2011}. In his paper, he extended the construction in \cite{BrzezinskiElKaoutitLomp2010} to finitely generated and projective modules.

Related to differential calculi, we have the {\em smoothness of algebras}. Briefly, the study of this smoothness goes back at least to Grothendieck's EGA \cite{Grothendieck1964}. The concept of a {\em formally smooth commutative} ({\em topological}) {\em algebra} introduced by him was extended to the noncommutative setting by Schelter \cite{Schelter1986}. An algebra is {\em formally smooth} if and only if the kernel of the multiplication map is projective as a bimodule. This notion arose as a replacement of a far too general definition based on the finiteness of the global dimension; Cuntz and Quillen \cite{CuntzQuillen1995} called these algebras {\em quasi-free}. Precisely, the notion of smoothness based on the finiteness of this dimension was refined by Stafford and Zhang \cite{StaffordZhang1994}, where a Noetherian algebra is said to be {\em smooth} provided that it has a finite global dimension equal to the homological dimension of all its simple modules. In the homological setting, Van den Bergh \cite{VandenBergh1998} called an algebra {\em homologically smooth} if it admits a finite resolution by finitely generated projective bimodules. The characterization of this kind of smoothness for the noncommutative pillow, the quantum teardrops, and quantum homogeneous spaces was made by Brzezi{\'n}ski \cite{Brzezinski2008, Brzezinski2014} and Kr\"ahmer \cite{Krahmer2012}, respectively.

Brzezi{\'n}ski and Sitarz \cite{BrzezinskiSitarz2017} defined other notion of smoothness of algebras, termed {\em differential smoothness} due to the use of differential graded algebras of a specified dimension that admits a noncommutative version of the Hodge star isomorphism, which considers the existence of a top form in a differential calculus over an algebra together with a string version of the Poincar\'e duality realized as an isomorphism between complexes of differential and integral forms. This new notion of smoothness is different and more constructive than the homological smoothness mentioned above. \textquotedblleft The idea behind the {\em differential smoothness} of algebras is rooted in the observation that a classical smooth orientable manifold, in addition to de Rham complex of differential forms, admits also the complex of {\em integral forms} isomorphic to the de Rham complex \cite[Section 4.5]{Manin1997}. The de Rham differential can be understood as a special left connection, while the boundary operator in the complex of integral forms is an example of a {\em right connection}.\textquotedblright\ \cite[p. 413]{BrzezinskiSitarz2017}

Several authors (e.g. \cite{Brzezinski2015, Brzezinski2016, BrzezinskiElKaoutitLomp2010, BrzezinskiLomp2018, BrzezinskiSitarz2017, DuboisVioletteKernerMadore1990, Karacuha2015, KaracuhaLomp2014, ReyesSarmiento2022}) have characterized the differential smoothness of algebras such as the quantum two - and three - spheres, disc, plane, the noncommutative torus, the coordinate algebras of the quantum group $SU_q(2)$, the noncommutative pillow algebra, the quantum cone algebras, the quantum polynomial algebras, Hopf algebra domains of Gelfand-Kirillov dimension two that are not PI, families of Ore extensions, some 3-dimensional skew polynomial algebras, diffusion algebras in three generators, and noncommutative coordinate algebras of deformations of several examples of classical orbifolds such as the pillow orbifold, singular cones and lens spaces. An interesting fact is that some of these algebras are also homologically smooth in the Van den Bergh's sense.

Motivated by the active research on differential smoothness of algebras, our purpose in this paper is to investigate this smoothness for the {\em bi-quadratic algebras on three generators with PBW basis} defined recently by Bavula \cite{Bavula2023, BavulaAlKhabyah2023}. As we can see from its definition, many of these are {\em generalized Weyl algebras} and {\em diskew polynomial rings} (see Bavula \cite{Bavula2020} for more details). Bi-quadratic algebras with PBW basis are related to noncommutative algebras of polynomial type defined and investigated previously by different authors (e.g. Apel \cite{Apel1988}, Bell et al. \cite{BellGoodearl1988, BellSmith1990}, Bueso et al. \cite{BuesoTorrecillasVerschoren2003}, Hinchcliffe \cite{Hinchcliffe2005}, Levandovskyy \cite{Levandovskyy2005}, Lezama et al. \cite{Fajardoetal2020, GallegoLezama2010, LezamaReyes2014, ReyesSuarez2017}, Li \cite{Li2002}, McConnell and Robson \cite{McConnellRobson2001}, Pyatov et al. \cite{IsaevPyatovRittenberg2001, PyatovTwarock2002}, Redman \cite{RedmanPhDThesis1996, Redman1999}, Rosenberg \cite{Rosenberg1995} and Seiler \cite{SeilerBook2010}). The results presented here contribute to the characterization of the smooth geometry of noncommutative algebras from Brzezi{\'n}ski and Sitarz's point of view.

The article is organized as follows. In Section \ref{PreliminariesDifferentialsmoothnessofbi-quadraticalgebras} we consider the necessary preliminaries on differential smoothness of algebras and bi-quadratic algebras with PBW basis. Section \ref{Differentialandintegralcalculusbi-quadraticalgebras} contains the original results on the paper. We extend Brzezi{\'n}ski's ideas developed for skew polynomial rings of the commutative polynomial ring $\Bbbk[t]$ \cite{Brzezinski2015}  to the setting of bi-quadratic algebras on three generators with PBW basis. More exactly, in Theorem \ref{Firsttheoremsmoothnessbi-quadraticalgebras} we formulate sufficient conditions to guarantee that a bi-quadratic algebra on three generators with PBW basis is differentially smooth, while Theorem \ref{Secondtheoremsmoothnessbi-quadraticalgebras} presents sufficient conditions on the impossibility of this fact. Following Bavula's classification of these algebras presented in his paper \cite{Bavula2023}, we organize our results using different tables to improve its presentation. Finally, in Section \ref{FutureworkDifferentialsmoothnessofbiquadraticalgebras} we say a few words about possible future work related to the topic.

Throughout the paper, $\mathbb{N}$ denotes the set of natural numbers including zero. The word ring means an associative ring with identity not necessarily commutative. $Z(R)$ denotes the center of the ring $R$. All vector spaces and algebras (always associative and with unit) are over a fixed field $\Bbbk$. As usual, the symbols $\mathbb{R}$ and $\mathbb{C}$ denote the fields of real and complex numbers, respectively. 

\section{Preliminaries}\label{PreliminariesDifferentialsmoothnessofbi-quadraticalgebras}

We start by recalling the preliminaries on differential smoothness of algebras and bi-quadratic algebras with PBW basis that are necessary for the rest of the paper.

\subsection{Differential smoothness of algebras}




We follow Brzezi\'nski and Sitarz's presentation on differential smoothness carried out in \cite[Section 2]{BrzezinskiSitarz2017} (c.f. \cite{Brzezinski2008, Brzezinski2014}).

\begin{definition}[{\cite[Section 2.1]{BrzezinskiSitarz2017}}]
A {\em differential graded algebra} is a non-negatively graded algebra $\Omega$ with the product denoted by $\wedge$ together with a degree-one linear map $d:\Omega^{\bullet} \to \Omega^{\bullet +1}$ that satisfies the graded Leibniz rule and is such that $d \circ d = 0$. 
\end{definition}

\begin{definition}[{\cite[Section 2.1]{BrzezinskiSitarz2017}}]
A differential graded algebra $(\Omega, d)$ is a {\em calculus over an algebra} $A$ if $\Omega^0 = A$ and $\Omega^n = A\ dA \wedge dA \wedge \dotsb \wedge dA$ ($dA$ appears $n$-times) for all $n\in \mathbb{N}$. We write $(\Omega A, d)$. By using the Leibniz rule, it follows that $\Omega^n A = dA \wedge dA \wedge \dotsb \wedge dA\ A$. A differential calculus $\Omega A$ is said to be {\em connected} if ${\rm ker}(d\mid_{\Omega_0})=\Bbbk$.
\end{definition}

\begin{definition}[{\cite[Section 2.1]{BrzezinskiSitarz2017}}]
A calculus $(\Omega A, d)$ is said to have {\em dimension} $n$ if $\Omega^n A\neq 0$ and $\Omega^m A = 0$ for all $m > n$. An $n$-dimensional calculus $\Omega A$ {\em admits a volume form} if $\Omega^n A$ is isomorphic to $A$ as a left and right $A$-module. 
\end{definition}

The existence of a right $A$-module isomorphism means that there is a free generator, say $\omega$, of $\Omega^n A$ (as a right $A$-module), i.e. $\omega \in \Omega^n A$, such that all elements of $\Omega^n A$ can be uniquely expressed as $\omega a$ with $a \in A$. If $\omega$ is also a free generator of $\Omega^n A$ as a left $A$-module, this is said to be a {\em volume form} on $\Omega A$.

The right $A$-module isomorphism $\Omega^n A \to A$ corresponding to a volume form $\omega$ is denoted by $\pi_{\omega}$, i.e.
\begin{equation}\label{BrzezinskiSitarz2017(2.1)}
\pi_{\omega} (\omega a) = a, \quad {\rm for\ all}\ a\in A.
\end{equation}

By using that $\Omega^n A$ is also isomorphic to $A$ as a left $A$-module, any free generator $\omega $ induces an algebra endomorphism $\nu_{\omega}$ of $A$ by the formula
\begin{equation}\label{BrzezinskiSitarz2017(2.2)}
    a \omega = \omega \nu_{\omega} (a).
\end{equation}

Note that if $\omega$ is a volume form, then $\nu_{\omega}$ is an algebra automorphism.

Now, we proceed to recall the key ingredients of the {\em integral calculus} on $A$ as dual to its differential calculus. For more details, see Brzezinski et al. \cite{Brzezinski2008, BrzezinskiElKaoutitLomp2010}.

Let $(\Omega A, d)$ be a differential calculus on $A$. The space of $n$-forms $\Omega^n A$ is an $A$-bimodule. Consider $\mathcal{I}_{n}A$ the right dual of $\Omega^{n}A$, the space of all right $A$-linear maps $\Omega^{n}A\rightarrow A$, that is, $\mathcal{I}_{n}A := {\rm Hom}_{A}(\Omega^{n}(A),A)$. Notice that each of the $\mathcal{I}_{n}A$ is an $A$-bimodule with the actions
\begin{align*}
    (a\cdot\phi\cdot b)(\omega)=a\phi(b\omega),\quad {\rm for\ all}\ \phi \in \mathcal{I}_{n}A,\ \omega \in \Omega^{n}A\ {\rm and}\ a,b \in A.
\end{align*}

The direct sum of all the $\mathcal{I}_{n}A$, that is, $\mathcal{I}A = \bigoplus\limits_{n} \mathcal{I}_n A$, is a right $\Omega A$-module with action given by
\begin{align}\label{BrzezinskiSitarz2017(2.3)}
    (\phi\cdot\omega)(\omega')=\phi(\omega\wedge\omega'),\quad {\rm for\ all}\ \phi\in\mathcal{I}_{n + m}A, \ \omega\in \Omega^{n}A \ {\rm and} \ \omega' \in \Omega^{m}A.
\end{align}

\begin{definition}[{\cite[Definition 2.1]{Brzezinski2008}}]
A {\em divergence} (also called {\em hom-connection}) on $A$ is a linear map $\nabla: \mathcal{I}_1 A \to A$ such that
\begin{equation}\label{BrzezinskiSitarz2017(2.4)}
    \nabla(\phi \cdot a) = \nabla(\phi) a + \phi(da), \quad {\rm for\ all}\ \phi \in \mathcal{I}_1 A \ {\rm and} \ a \in A.
\end{equation}  
\end{definition}

Note that a divergence can be extended to the whole of $\mathcal{I}A$, 
\[
\nabla_n: \mathcal{I}_{n+1} A \to \mathcal{I}_{n} A,
\]

by considering
\begin{equation}\label{BrzezinskiSitarz2017(2.5)}
\nabla_n(\phi)(\omega) = \nabla(\phi \cdot \omega) + (-1)^{n+1} \phi(d \omega), \quad {\rm for\ all}\ \phi \in \mathcal{I}_{n+1}(A)\ {\rm and} \ \omega \in \Omega^n A.
\end{equation}

By putting together (\ref{BrzezinskiSitarz2017(2.4)}) and (\ref{BrzezinskiSitarz2017(2.5)}), we get the Leibniz rule 
\begin{equation}
    \nabla_n(\phi \cdot \omega) = \nabla_{m + n}(\phi) \cdot \omega + (-1)^{m + n} \phi \cdot d\omega,
\end{equation}

for all elements $\phi \in \mathcal{I}_{m + n + 1} A$ and $\omega \in \Omega^m A$ \cite[Lemma 3.2]{Brzezinski2008}. In the case $n = 0$, if ${\rm Hom}_A(A, M)$ is canonically identified with $M$, then $\nabla_0$ reduces to the classical Leibniz rule.

\begin{definition}[{\cite[Definition 3.4]{Brzezinski2008}}]
The right $A$-module map 
$$
F = \nabla_0 \circ \nabla_1: {\rm Hom}_A(\Omega^{2} A, M) \to M
$$ is called a {\em curvature} of a hom-connection $(M, \nabla_0)$. $(M, \nabla_0)$ is said to be {\em flat} if its curvature is the zero map, that is, if $\nabla \circ \nabla_1 = 0$. This condition implies that $\nabla_n \circ \nabla_{n+1} = 0$ for all $n\in \mathbb{N}$.
\end{definition}

$\mathcal{I} A$ together with the $\nabla_n$ form a chain complex called the {\em complex of integral forms} over $A$. The cokernel map of $\nabla$, that is, $\Lambda: A \to {\rm Coker} \nabla = A / {\rm Im} \nabla$ is said to be the {\em integral on $A$ associated to} $\mathcal{I}A$.

Given a left $A$-module $X$ with action $a\cdot x$, for all $a\in A,\ x \in X$, and an algebra automorphism $\nu$ of $A$, the notation $^{\nu}X$ stands for $X$ with the $A$-module structure twisted by $\nu$, i.e. with the $A$-action $a\otimes x \mapsto \nu(a)\cdot x $.

The following definition of an \textit{integrable differential calculus} seeks to portray a version of Hodge star isomorphisms between the complex of differential forms of a differentiable manifold and a complex of dual modules of it \cite[p. 112]{Brzezinski2015}. 

\begin{definition}[{\cite[Definition 2.1]{BrzezinskiSitarz2017}}]
An $n$-dimensional differential calculus $(\Omega A, d)$ is said to be {\em integrable} if $(\Omega A, d)$ admits a complex of integral forms $(\mathcal{I}A, \nabla)$ for which there exist an algebra automorphism $\nu$ of $A$ and $A$-bimodule isomorphisms \linebreak $\Theta_k: \Omega^{k} A \to ^{\nu} \mathcal{I}_{n-k}A$, $k = 0, \dotsc, n$, rendering commmutative the following diagram:
\[
{\large{
\begin{tikzcd}
A \arrow{r}{d} \arrow{d}{\Theta_0} & \Omega^{1} A \arrow{d}{\Theta_1} \arrow{r}{d} & \Omega^2 A  \arrow{d}{\Theta_2} \arrow{r}{d} & \dotsb \arrow{r}{d} & \Omega^{n-1} A \arrow{d}{\Theta_{n-1}} \arrow{r}{d} & \Omega^n A  \arrow{d}{\Theta_n} \\ ^{\nu} \mathcal{I}_n A \arrow[swap]{r}{\nabla_{n-1}} & ^{\nu} \mathcal{I}_{n-1} A \arrow[swap]{r}{\nabla_{n-2}} & ^{\nu} \mathcal{I}_{n-2} A \arrow[swap]{r}{\nabla_{n-3}} & \dotsb \arrow[swap]{r}{\nabla_{1}} & ^{\nu} \mathcal{I}_{1} A \arrow[swap]{r}{\nabla} & ^{\nu} A
\end{tikzcd}
}}
\]

The $n$-form $\omega:= \Theta_n^{-1}(1)\in \Omega^n A$ is called an {\em integrating volume form}. 
\end{definition}

The algebra of complex matrices $M_n(\mathbb{C})$ with the $n$-dimensional calculus generated by derivations presented by Dubois-Violette et al. \cite{DuboisViolette1988, DuboisVioletteKernerMadore1990}, the quantum group $SU_q(2)$ with the three-dimensional left covariant calculus developed by Woronowicz \cite{Woronowicz1987} and the quantum standard sphere with the restriction of the above calculus, are examples of algebras admitting integrable calculi. For more details on the subject, see Brzezi\'nski et al. \cite{BrzezinskiElKaoutitLomp2010}. 

The following proposition shows that the integrability of a differential calculus can be defined without explicit reference to integral forms. This allows us to guarantee the integrability by considering the existence of finitely generator elements that allow to determine left and right components of any homogeneous element of $\Omega(A)$.

\begin{proposition}[{\cite[Theorem 2.2]{BrzezinskiSitarz2017}}]\label{integrableequiva} 
Let $(\Omega A, d)$ be an $n$-dimensional differential calculus over an algebra $A$. The following assertions are equivalent:
\begin{enumerate}
    \item [\rm (1)] $(\Omega A, d)$ is an integrable differential calculus.
    
    \item [\rm (2)] There exists an algebra automorphism $\nu$ of $A$ and $A$-bimodule isomorphisms $\Theta_k : \Omega^k A \rightarrow \ ^{\nu}\mathcal{I}_{n-k}A$, $k =0, \ldots, n$, such that, for all $\omega'\in \Omega^k A$ and $\omega''\in \Omega^mA$,
    \begin{align*}
        \Theta_{k+m}(\omega'\wedge\omega'')=(-1)^{(n-1)m}\Theta_k(\omega')\cdot\omega''.
    \end{align*}
    
    \item [\rm (3)] There exists an algebra automorphism $\nu$ of $A$ and an $A$-bimodule map $\vartheta:\Omega^nA\rightarrow\ ^{\nu}A$ such that all left multiplication maps
    \begin{align*}
    \ell_{\vartheta}^{k}:\Omega^k A &\ \rightarrow \mathcal{I}_{n-k}A, \\
    \omega' &\ \mapsto \vartheta\cdot\omega', \quad k = 0, 1, \dotsc, n,
    \end{align*}
    where the actions $\cdot$ are defined by {\rm (}\ref{BrzezinskiSitarz2017(2.3)}{\rm )}, are bijective.
    
    \item [\rm (4)] $(\Omega A, d)$ has a volume form $\omega$ such that all left multiplication maps
    \begin{align*}
        \ell_{\pi_{\omega}}^{k}:\Omega^k A &\ \rightarrow \mathcal{I}_{n-k}A, \\
        \omega' &\ \mapsto \pi_{\omega} \cdot \omega', \quad k=0,1, \dotsc, n-1,
    \end{align*}
    
    where $\pi_{\omega}$ is defined by {\rm (}\ref{BrzezinskiSitarz2017(2.1)}{\rm )}, are bijective.
\end{enumerate}
\end{proposition}

A volume form $\omega\in \Omega^nA$ is an {\em integrating form} if and only if it satisfies condition $(4)$ of Proposition \ref{integrableequiva} \cite[Remark 2.3]{BrzezinskiSitarz2017}.

The most interesting cases of differential calculi are those where $\Omega^k A$ are finitely generated and projective right or left (or both) $A$-modules \cite{Brzezinski2011}.

\begin{proposition}\label{BrzezinskiSitarz2017Lemmas2.6and2.7}
\begin{enumerate}
\item [\rm (1)] \cite[Lemma 2.6]{BrzezinskiSitarz2017} Consider $(\Omega A, d)$ an integrable and $n$-dimensional calculus over $A$ with integrating form $\omega$. Then $\Omega^{k} A$ is a finitely generated projective right $A$-module if there exist a finite number of forms $\omega_i \in \Omega^{k} A$ and $\overline{\omega}_i \in \Omega^{n-k} A$ such that, for all $\omega' \in \Omega^{k} A$, we have that 
\begin{equation*}
\omega' = \sum_{i} \omega_i \pi_{\omega} (\overline{\omega}_i \wedge \omega').
\end{equation*}

\item [\rm (2)] \cite[Lemma 2.7]{BrzezinskiSitarz2017} Let $(\Omega A, d)$ be an $n$-dimensional calculus over $A$ admitting a volume form $\omega$. Assume that for all $k = 1, \ldots, n-1$, there exists a finite number of forms $\omega_{i}^{k},\overline{\omega}_{i}^{k} \in \Omega^{k}(A)$ such that for all $\omega'\in \Omega^kA$, we have that
\begin{equation*}
\omega'=\displaystyle\sum_i\omega_{i}^{k}\pi_\omega(\overline{\omega}_{i}^{n-k}\wedge\omega')=\displaystyle\sum_i\nu_{\omega}^{-1}(\pi_\omega(\omega'\wedge\omega_{i}^{n-k}))\overline{\omega}_{i}^{k},
\end{equation*}

where $\pi_{\omega}$ and $\nu_{\omega}$ are defined by {\rm (}\ref{BrzezinskiSitarz2017(2.1)}{\rm )} and {\rm (}\ref{BrzezinskiSitarz2017(2.2)}{\rm )}, respectively. Then $\omega$ is an integral form and all the $\Omega^{k}A$ are finitely generated and projective as left and right $A$-modules.
\end{enumerate}
\end{proposition}

Brzezi\'nski and Sitarz \cite[p. 421]{BrzezinskiSitarz2017} asserted that to connect the integrability of the differential graded algebra $(\Omega A, d)$ with the algebra $A$, it is necessary to relate the dimension of the differential calculus $\Omega A$ with that of $A$, and since we are dealing with algebras that are deformations of coordinate algebras of affine varieties, the {\em Gelfand-Kirillov dimension} introduced by Gelfand and Kirillov \cite{GelfandKirillov1966, GelfandKirillov1966b} seems to be the best suited. Briefly, given an affine $\Bbbk$-algebra $A$, the {\em Gelfand-Kirillov dimension of} $A$, denoted by ${\rm GKdim}(A)$, is given by
\[
{\rm GKdim}(A) := \underset{n\to \infty}{\rm lim\ sup} \frac{{\rm log}({\rm dim}\ V^{n})}{{\rm log}\ n},
\]

where $V$ is a finite-dimensional subspace of $A$ that generates $A$ as an algebra. This definition is independent of choice of $V$. If $A$ is not affine, then its Gelfand-Kirillov dimension is defined to be the supremum of the Gelfand-Kirillov dimensions of all affine subalgebras of $A$. An affine domain of Gelfand-Kirillov dimension zero is precisely a division ring that is finite-dimensional over its center. In the case of an affine domain of Gelfand-Kirillov dimension one over $\Bbbk$, this is precisely a finite module over its center, and thus polynomial identity. In some sense, this dimensions measures the deviation of the algebra $A$ from finite dimensionality. For more details about this dimension, see the excellent treatment developed by Krause and Lenagan \cite{KrauseLenagan2000}.

After preliminaries above, we arrive to the key notion of this paper.

\begin{definition}[{\cite[Definition 2.4]{BrzezinskiSitarz2017}}]\label{BrzezinskiSitarz2017Definition2.4}
An affine algebra $A$ with integer Gelfand-Kirillov dimension $n$ is said to be {\em differentially smooth} if it admits an $n$-dimensional connected integrable differential calculus $(\Omega A, d)$.
\end{definition}

From Definition \ref{BrzezinskiSitarz2017Definition2.4} it follows that a differentially smooth algebra comes equipped with a well-behaved differential structure and with the precise concept of integration \cite[p. 2414]{BrzezinskiLomp2018}.

\begin{example}
As we said in the Introduction, several examples of noncommutative algebras have been proved to be differentially smooth (e.g. \cite{Brzezinski2015, BrzezinskiElKaoutitLomp2010, BrzezinskiLomp2018, BrzezinskiSitarz2017, Karacuha2015, KaracuhaLomp2014, ReyesSarmiento2022}). For instance, the polynomial algebra $\mathbb{C}[x_1, \dotsc, x_n]$ has the Gelfand-Kirillov dimension $n$ and the usual exterior algebra is an $n$-dimensional integrable calculus, whence $\mathbb{C}[x_1, \dotsc, x_n]$ is differentially smooth. From \cite{BrzezinskiElKaoutitLomp2010} it follows that the coordinate algebras of the quantum group $SU_q(2)$, the standard quantum Podle\'s and the quantum Manin plane are differentially smooth.
\end{example}

There are examples of algebras that are not differentially smooth. To illustrate this fact, take the algebra $A = \mathbb{C}[x, y] / \langle xy \rangle$. A proof by contradiction shows that for this algebra there are no one-dimensional connected integrable calculi over $A$, so it cannot be differentially smooth \cite[Example 2.5]{BrzezinskiSitarz2017}.

\subsection{Bi-quadratic algebras with PBW basis}\label{BiquadraticalgebrasPBWbasis}

For a natural number $n\ge 2$, a family $M = (m_{ij})_{i > j}$ of elements $m_{ij}$ belonging to $R$ ($1\le j < i \le n$) is called a {\em lower triangular half-matrix} with coefficients in $R$. The set of all such matrices is denoted by $L_n(R)$.

\begin{definition}[{\cite[Section 1]{Bavula2023}}]
If $\sigma = (\sigma_1, \dotsc, \sigma_n)$ is an $n$-tuple of commuting endomorphisms of $R$, $\delta = (\delta_1, \dotsc, \delta_n)$ is an $n$-tuple of $\sigma$-endomorphisms of $R$ (that is, $\delta_i$ is a $\sigma_i$-derivation of $R$ for $i=1,\dotsc, n$), $Q = (q_{ij})\in L_n(Z(R))$, $\mathbb{A}:= (a_{ij, k})$ where $a_{ij, k}\in R$, $1\le j < i \le n$ and $k = 1,\dotsc, n$, and $\mathbb{B}:= (b_{ij})\in L_n(R)$, the {\em skew bi-quadratic algebra} ({\em SBQA}) $A = R[x_1,\dotsc, x_n;\sigma, \delta, Q, \mathbb{A}, \mathbb{B}]$ is a ring generated by the ring $R$ and elements $x_1, \dotsc, x_n$ subject to the defining relations
\begin{align}
    x_ir = &\ \sigma_i(r)x_i + \delta_i(r),\quad {\rm for}\ i = 1, \dotsc, n,\ {\rm and\ every}\ r\in R, \label{Bavula2023(1)} \\
    x_ix_j - q_{ij}x_jx_i = &\ \sum_{k=1}^{n} a_{ij, k}x_k + b_{ij},\quad {\rm for\ all}\ j < i.\label{Bavula2023(2)}
\end{align}

If $\sigma_i = {\rm id}_R$ and $\delta_i = 0$ for $i = 1,\dotsc, n$, the ring $A$ is called the {\em bi-quadratic algebra} ({\em BQA}) and is denoted by $A = R[x_1, \dotsc, x_n; Q, \mathbb{A}, \mathbb{B}]$. $A$ has {\em PBW basis} if $A = \bigoplus\limits_{\alpha \in \mathbb{N}^{n}} Rx^{\alpha}$ where $x^{\alpha} = x_1^{\alpha_1}\dotsb x_n^{\alpha_n}$.
\end{definition}

On the set $W_n$ of all words in the alphabet $\{x_1, \dotsc, x_n\}$, Bavula considered the {\em degree-by-lexicographic ordering} where $x_1 < \dotsb < x_n$. More exactly, $x_{i_1} \dotsb x_{i_s}$ if either $s < t$ or $s = t,\ i_1 = j_1, \dotsc, i_k = j_k$ and $i_{k+1} < j_{k+1}$ for some $k$ such that $1 \le k < s$. Hence, if $A = R[x_1, \dotsc, x_n; Q, \mathbb{A}, \mathbb{B}]$ is a quadratic algebra where $n\ge 3$, then for each triple $i, j, k \in \{1, \dotsc, n\}$ such that $ i < j < k$, there are exactly two different ways to simplify the product $x_kx_jx_i$ with respect to this order given by
\begin{align*}
    x_k x_j x_i = &\ q_{kj} q_{ki} q_{ji} x_i x_j x_k + \sum_{|\alpha| \le 2} c_{k, j, i, \alpha} x^{\alpha}, \\
    x_k x_j x_i = &\ q_{kj} q_{ki} q_{ji} x_i x_j x_k + \sum_{|\alpha| \le 2} c'_{k, j, i, \alpha} x^{\alpha},
\end{align*}

where in the first (resp. second) equality we start to simplify the product with the relation $x_k x_j = q_{kj} x_j x_k + \dotsb$ (resp. $x_j x_i = q_{ji} x_i x_j + \dotsb$) \cite[p. 696]{Bavula2023}. Thus, the defining relations (\ref{Bavula2023(2)}) are consistent (i.e. $A\neq 0$) and $A = \bigoplus\limits_{\alpha \in \mathbb{N}^n} Rx^{\alpha}$ if and only if for all triples $i, j, k \in \{1, \dotsc, n\}$ such that $i < j < k$, we have that $ c_{k, j, i, \alpha} = c'_{k, j, i, \alpha}$. If this is the case, then for all $\sigma \in S_n$ ($S_n$ denotes the {\em symmetric group of order $n$}) we have that $A = \bigoplus\limits_{\alpha \in \mathbb{N}^n} Rx_{\sigma}^{\alpha}$ where $x_{\sigma}^{\alpha} = x_{\sigma(1)}^{\alpha_1} \dotsb x_{\sigma(n)}^{\alpha_n}$ and $\alpha = (\alpha_1, \dotsc, \alpha_n)$ \cite[Theorem 1.1]{Bavula2023}.

\begin{remark}\label{BiquadraticalgebrasPBWbasisTwogenerators}
The classification of bi-quadratic algebras on two generators over a field $\Bbbk$ is as follows: if $q\in \Bbbk^{*}$ and $a, b, c \in \Bbbk$, then the algebra
\[
A = \Bbbk[x_1, x_2; q, a, b, c] := \Bbbk\{ x_1, x_2\} / \langle x_2 x_1 - qx_1x_2 - ax_1 - bx_2 - c\rangle
\]

is a bi-quadratic algebra on two generators \cite[p. 704]{Bavula2023}. The algebra $A = \Bbbk[x_1][x_2; \sigma, \delta]$ is a {\em skew polynomial algebra} or {\em Ore extension} (introduced by Ore \cite{Ore1931, Ore1933}) of $\Bbbk[x_1]$ where $\sigma(x_1) = qx_1 + b$ and $\delta(x_1) = ax_1 + c$. In this way, the algebra $A$ is a Noetherian domain with PBW basis, $A = \bigoplus\limits_{i, j \in \mathbb{N}} \Bbbk x_1^{i} x_2^{j}$ and the scalars $a, b, c$ are arbitrary.

Up to isomorphism, there are only five bi-quadratic algebras on two generators \cite[Theorem 2.1]{Bavula2023}: 
    \begin{itemize}
        \item The {\em polynomial algebra} $\Bbbk[x_1, x_2]$;
        
        \item The {\em Weyl algebra} $A_1(\Bbbk) = \Bbbk\{x_1, x_2\} / \langle x_1x_2 - x_2x_1 - 1\rangle$;
        
        \item The {\em universal enveloping algebra of the Lie algebra} $\mathfrak{n}_2 = \langle x_1, x_2\mid [x_2, x_1] = x_1\rangle$, that is, $U(\mathfrak{n}_2) = \Bbbk\{x_1, x_2\} / \langle x_2x_1 - x_1x_2 - x_1\rangle$;
        
        \item The {\em quantum plane} ({\em Manin's plane}) $\mathcal{O}_q(\Bbbk) = \Bbbk \{x_1, x_2\} / \langle x_2 x_1 - qx_1 x_2\rangle$, where $q\in \Bbbk\ \backslash\ \{0,1\}$;
        
        \item The {\em quantum Weyl algebra} $A_1(q) = \Bbbk \{x_1, x_2\} / \langle x_2x_1 - qx_1x_2 - 1\rangle$, where $q\in \Bbbk\ \backslash\ \{0,1\}$.
    \end{itemize}
\end{remark}

Proposition \ref{Bavula2023Theorem1.3} gives necessary and sufficient conditions for the bi-quadratic algebras on three-generators have a PBW basis. As we will see in Section \ref{DSBiquadraticalgebrasThreegenerators}, this is key to characterize the differential smoothness of these algebras.

\begin{proposition}[{\cite[Theorem 1.3]{Bavula2023}}]\label{Bavula2023Theorem1.3}
Let $A = \Bbbk[x_1, x_2, x_3;Q, \mathbb{A}, \mathbb{B}]$ be a bi-quadratic algebra where $Q = (q_1, q_2, q_3) \in \Bbbk^{* 3}$,
\[
\mathbb{A} = \begin{bmatrix}
    a & b & c \\ \alpha & \beta & \gamma \\ \lambda & \mu & \nu
\end{bmatrix}\quad {\rm and}\quad \mathbb{B} = \begin{bmatrix} b_1 \\ b_2 \\ b_3 \end{bmatrix}.
\]

The algebra $A$ is generated over $\Bbbk$ by the elements $x_1, x_2$ and $x_3$ subject to the defining relations
\begin{align}
    x_2 x_1 - q_1x_1x_2 = &\ ax_1 + bx_2 + cx_3 + b_1, \label{Bavula2023(8)} \\
    x_3 x_1 - q_2 x_1 x_3 = &\ \alpha x_1 + \beta x_2 + \gamma x_3 + b_2, \label{Bavula2023(9)} \\
    x_3 x_2 - q_3x_2 x_3 = &\ \lambda x_1 + \mu x_2 + \nu x_3 + b_3. \label{Bavula2023(10)}
\end{align}

The relations {\rm(}\ref{Bavula2023(8)}{\rm )}, {\rm (}\ref{Bavula2023(9)}{\rm )} and {\rm (}\ref{Bavula2023(10)}{\rm )} are consistent and $A = \bigoplus\limits_{\alpha \in \mathbb{N}^3} \Bbbk x^{\alpha}$ where $x^{\alpha} = x_1^{\alpha_1} x_2^{\alpha_2}x_3^{\alpha_3}$ if and only if the following conditions hold:
\begin{align}
&\ (1 - q_3)\alpha =  (1 - q_2)\mu, \label{Bavula2023(11)} \\
&\ (1 - q_3) a = (1 - q_1)\nu, \label{Bavula2023(12)} \\
&\ (1 - q_2) b = (1 - q_1)\gamma \label{Bavula2023(13)} \\
&\ (1 - q_1 q_2) \lambda = 0, \label{Bavula2023(14)} \\
&\ (q_1 - q_3) \beta = 0, \label{Bavula2023(15)} \\
&\ (1 - q_2 q_3)c = 0, \label{Bavula2023(16)} \\
&\ ((1-q_3)\alpha -\mu)a+(b+q_1\gamma)\lambda-\nu\alpha+(q_1q_2-1)b_3=0, \label{Bavula2023(17)} \\
&\ (a-\nu)\beta + q_1\gamma\mu -q_3\alpha b+(q_1-q_3)b_2=0, \label{Bavula2023(18)}\\
&\ (a+(q_1-1)\nu)\gamma+b\nu-(\mu+q_3\alpha)c+(1-q_2q_3)b_1=0, \label{Bavula2023(19)}\\
&\ -(\mu + q_3\alpha)b_1+(a-\nu)b_2+(b+q_1\gamma)b_3=0. \label{Bavula2023(20)}
\end{align}

Furthermore, if $A = \bigoplus\limits_{\alpha \in \mathbb{N}^3} \Bbbk x^{\alpha}$ where $x^{\alpha} = x_1^{\alpha_1} x_2^{\alpha_2} x_3^{\alpha_3}$ then $A = \bigoplus\limits_{\alpha \in \mathbb{N}^3} \Bbbk x^{\alpha}_{\sigma}$ for all $\sigma \in S_3$ where $x^{\alpha}_{\sigma} = x^{\alpha_1}_{\sigma(1)} x^{\alpha_2}_{\sigma(2)} x^{\alpha_3}_{\sigma(3)}$.
\end{proposition}

Related noncommutative algebras on three generators and PBW basis have been defined by Bell and Smith \cite{BellSmith1990} (see also Rosenberg \cite[Theorem C4.3.1]{Rosenberg1995}), and Pyatov et al. \cite{IsaevPyatovRittenberg2001, PyatovTwarock2002}.

From now on, the expression \textquotedblleft bi-quadratic algebra\textquotedblright\ means \textquotedblleft bi-quadratic algebra with PBW basis\textquotedblright.

\begin{example}
Let us see some examples of bi-quadratic algebras on three generators.
\begin{enumerate}
    \item [\rm (a)] The universal enveloping algebra of any 3-dimensional Lie algebra.
    
    \item [\rm (b)] The 3-{\em dimensional quantum space} $\mathbb{A}_{q_1, q_2, q_3}^{3} := \Bbbk[x_1, x_2, x_3; Q, \mathbb{A} = 0, \mathbb{B} = 0]$.
    
    \item [\rm (c)] Following Havli\v{c}ek et al. \cite[p. 79]{HavlicekKlimykPosta2000}, the $ \Bbbk$-algebra $U_q'(\mathfrak{so}_3)$ is generated by the indeterminates $I_1, I_2$, and $I_3$ subject to the relations given by
\begin{align*}
    I_2I_1 - qI_1I_2 = &\ -q^{\frac{1}{2}}I_3, \\
    I_3I_1 - q^{-1}I_1I_3 = &\ q^{-\frac{1}{2}}I_2, \quad {\rm and} \\
    I_3I_2 - qI_2I_3 = &\ -q^{\frac{1}{2}}I_1, \quad {\rm for}\ q\in \Bbbk\ \backslash \ \{0, \pm 1\}. 
\end{align*}

\item [\rm (d)] Zhedanov \cite[Section 1]{Zhedanov1991} introduced the {\em Askey-Wilson algebra} $AW(3)$ as the $\mathbb{R}$-algebra generated by three operators $K_0, K_1$, and $K_2$, that satisfy the commutation relations 
\begin{align*}
[K_0, K_1]_{\omega} = &\ K_2, \\
[K_2, K_0]_{\omega} = &\ BK_0 + C_1K_1 + D_1, \quad {\rm and} \\
[K_1, K_2]_{\omega} = &\ BK_1 + C_0K_0 + D_0, 
\end{align*}

where $B, C_0, C_1, D_0$, and $D_1$ are elements of $\mathbb{R}$ that represent the structure constants of the algebra, and the $q$-commutator $[ - , -]_{\omega}$ is given by $[\square, \triangle]_{\omega}:= e^{\omega}\square \triangle - e^{- \omega}\triangle \square$, where $\omega\in \mathbb{R}$. Notice that in the limit $\omega \to 0$, the algebra AW(3) becomes an ordinary Lie algebra with three generators ($D_0$ and $D_1$ are included among the structure constants of the algebra in order to take into account algebras of Heisenberg-Weyl type). The relations defining the algebra can be written as 
\begin{align*}
    e^{\omega}K_0K_1 - e^{-\omega}K_1K_0 = &\ K_2,\\
    e^{\omega} K_2K_0 - e^{-\omega}K_0 K_2 = &\ BK_0 + C_1K_1 + D_1, \quad {\rm and} \\
    e^{\omega}K_1K_2 - e^{-\omega}K_2K_1 = &\ BK_1 + C_0K_0 + D_0.
\end{align*}
\end{enumerate}
\end{example}

With the aim of classifying bi-quadratic algebras on three generators, Bavula \cite{Bavula2023, BavulaAlKhabyah2023} considered $Q = (q_1, q_2, q_3) \in \Bbbk^{* 3}$ in Proposition \ref{Bavula2023Theorem1.3} into the following four cases:
\begin{itemize}
    \item $q_1 = q_2 = q_3 = 1$ (Lie type); 
    \item $q_1 \neq 1,\ q_2 = q_3 = 1$; 
    \item $q_1\neq 1,\ q_2\neq 1,\ q_3 = 1$; and
    \item $q_1\neq 1,\ q_2\neq 1,\ q_3\neq 1$.
\end{itemize}

We will also consider these cases in the characterization of the differential smoothness of these algebras presented in the next section. As it can be seen from Bavula's papers, there are exactly $44$ types (up to isomorphism and considering $\sqrt{\Bbbk}\subseteq \Bbbk$ and $\sqrt[3]{\Bbbk}\subseteq \Bbbk$) of non-isomorphic algebras of bi-quadratic algebras on three generators (see Tables \ref{FirsttableDSBiquadraticalgebras3}, \ref{SecondtableDSBiquadraticalgebras3}, \ref{ThirdtableDSBiquadraticalgebras3}, \ref{FourthtableDSBiquadraticalgebras3}, \ref{FifthtableDSBiquadraticalgebras3}, \ref{SixthtableDSBiquadraticalgebras3}, \ref{SeventhtableDSBiquadraticalgebras3} and \ref{EighttableDSBiquadraticalgebras3} for the details of each algebra). 

\section{Differential and integral calculus}\label{Differentialandintegralcalculusbi-quadraticalgebras}

In this section we investigate the differential smoothness of bi-quadratic algebras on three generators. Before, we say a few words about the smoothness of the bi-quadratic algebras on two generators.

\subsection{Bi-quadratic algebras on two generators}\label{DSBiquadraticalgebrasTwogenerators}

Brzezi{\'n}ski \cite{Brzezinski2015} characterized the differential smoothness of skew polynomial rings of the form $\Bbbk[t][x; \sigma_{q, r}, \delta_{p(t)}]$ where $\sigma_{q, r}(t) = qt + r$, with $q, r \in \Bbbk,\ q\neq 0$, and the $\sigma_{q, r}-$derivation $\delta_{p(t)}$ is defined as
\[
\delta_{p(t)} (f(t)) = \frac{f(\sigma_{q, r}(t)) - f(t)}{\sigma_{q, r}(t) - t} p(t),
\]

for an element $p(t) \in \Bbbk[t]$. $\delta_{p(t)}(f(t))$ is a suitable limit when $q = 1$ and $r = 0$, that is, when $\sigma_{q, r}$ is the identity map of $\Bbbk[t]$.

For the maps
\begin{equation}\label{Brzezinski2015(3.4)}
\nu_t(t) = t,\quad \nu_t(x) = qx + p'(t)\quad {\rm and}\quad \nu_x(t) = \sigma_{q, r}^{-1}(t),\quad \nu_x(x) = x,
\end{equation}

where $p'(t)$ is the classical $t$-derivative of $p(t)$, Brzezi{\'n}ski \cite[Lemma 3.1]{Brzezinski2015} showed that all of them simultaneously extend to algebra automorphisms $\nu_t$ and $\nu_x$ of $\Bbbk[t][x; \sigma_{q, r}, \delta_{p(t)}]$ only in the following three cases:
    \begin{enumerate}
        \item [\rm (a)] $q = 1, r = 0$ with no restriction on $p(t)$;
        
        \item [\rm (b)] $q = 1, r\neq 0$ and $p(t) = c$, $c\in \Bbbk$;
        
        \item [\rm (c)] $q\neq 1, p(t) = c\left( t + \frac{r}{q-1} \right)$, $c\in \Bbbk$ with no restriction on $r$.
    \end{enumerate}
    
In any of the cases {\rm (a) - (c)} we have that $\nu_x \circ \nu_t = \nu_t \circ \nu_x$. If the Ore extension $\Bbbk[t][x; \sigma_{q, r}, \delta_{p(t)}]$ satisfies one of these three conditions, then it is differentially smooth \cite[Proposition 3.3]{Brzezinski2015}.

As we saw in Remark \ref{BiquadraticalgebrasPBWbasisTwogenerators}, up to isomorphism, there are only five bi-quadratic algebras on two generators, so that having in mind Brzezi{\'n}ski's result on the differential smoothness of $\Bbbk[t][x; \sigma_{q, r}, \delta_{p(t)}]$, we get that the algebras $\Bbbk[x_1,x_2],\ A_1(\Bbbk)$ and $U(\mathfrak{n}_2)$ belong to the case (a), while $\mathcal{O}_q(\Bbbk)$ belongs to the case (c), whence these four algebras are differentially smooth. With respect to the quantum Weyl algebra $A_1(q)$, this cannot be classified by using this kind of skew polynomial algebras so that its differential smoothness should be investigated by considering other and alternative approach.

\subsection{Bi-quadratic algebras on three generators}\label{DSBiquadraticalgebrasThreegenerators}

Throughout this section, $A$ denotes a bi-quadratic algebra on three generators $x_1$, $x_2$ and $x_3$ subject to the relations {\rm (}\ref{Bavula2023(8)}{\rm )}, {\rm (}\ref{Bavula2023(9)}{\rm )} and {\rm (}\ref{Bavula2023(10)}{\rm )} in Proposition \ref{Bavula2023Theorem1.3}. Since these are a subclass of the {\em skew PBW extensions} introduced by Gallego and Lezama \cite{GallegoLezama2010}, it follows from \cite[Theorems 14 and 18]{Reyes2013} that its Gelfand-Kirillov dimension is three. This fact is key in Theorems \ref{Firsttheoremsmoothnessbi-quadraticalgebras} and \ref{Secondtheoremsmoothnessbi-quadraticalgebras}.

The following theorem is the first important result of the paper. We extend Brzezi\'nski's ideas \cite{Brzezinski2015} mentioned in Section \ref{DSBiquadraticalgebrasTwogenerators}.

\begin{theorem}\label{Firsttheoremsmoothnessbi-quadraticalgebras}
If the conditions
\begin{align}
    c = \beta = \lambda = &\ 0, \\
    b_1(q_1-1)-ab = &\ 0, \\
    b_2(q_2-1)-\alpha\gamma = &\ 0, \\
    b_3(q_3-1) - \mu \nu = &\ 0, \quad {\rm and} \\
    \mu a = &\ 0
\end{align}

hold, then $A$ is differentially smooth.
\end{theorem}
\begin{proof}
Consider the following automorphisms:
\begin{align}
   \nu_{x_1}(x_1) = &\ x_1, & \nu_{x_1}(x_2) = &\ q_1x_2+a, & \nu_{x_1}(x_3) = &\ q_2x_3+\alpha, \label{Auto1} \\ 
    \nu_{x_2}(x_1) = &\ q_1^{-1}x_1-bq_1^{-1}, & \nu_{x_2}(x_2) = &\ x_2, &  \nu_{x_2}(x_3) = &\ q_3x_3+\mu, \label{Auto2} \\
    \nu_{x_3}(x_1) = &\ q_2^{-1}x_1-\gamma q_2^{-1}, & \nu_{x_3}(x_2) = &\ q_3^{-1}x_2-\nu q_3^{-1}, & \nu_{x_3}(x_3) = &\ x_3. \label{Auto3}
\end{align}

The map $\nu_{x_1}$ can be extended to an algebra homomorphism of $A$ if and only if the definitions of $\nu_{x_1}(x_1)$, $\nu_{x_1}(x_2)$ and $\nu_{x_1}(x_3)$ respect relations  {\rm (}\ref{Bavula2023(8)}{\rm )}, {\rm (}\ref{Bavula2023(9)}{\rm )} and {\rm (}\ref{Bavula2023(10)}{\rm )}, i.e.
\begin{align*}
   \nu_{x_1}(x_2)\nu_{x_1}(x_1)-q_1\nu_{x_1}(x_1)\nu_{x_1}(x_2) = &\ a\nu_{x_1}(x_1) + b\nu_{x_1}(x_2)+b_1, \\
   \nu_{x_1}(x_3)\nu_{x_1}(x_1)-q_2\nu_{x_1}(x_1)\nu_{x_1}(x_3) = &\ \alpha\nu_{x_1}(x_1) +\gamma\nu_{x_1}(x_3)+b_2, \quad {\rm and} \\
  \nu_{x_1}(x_3)\nu_{x_1}(x_2)-q_3\nu_{x_1}(x_2)\nu_{x_1}(x_3) = &\ \mu\nu_{x_1}(x_2) + \nu\nu_{x_1}(x_3)+b_3.
\end{align*}

In this way, we obtain the equations 
\begin{align*}
     b_1(q_1-1)-ab = &\ 0,\\
     b_2(q_2-1)-\alpha\gamma = &\ 0, \\
     \mu(q_2-1) = &\ \alpha(q_3-1),\\  \nu(q_1-1) = &\ a(q_3-1), \quad {\rm and} \\ 
     b_3(q_1q_2-1) = &\ \alpha(a(q_3-1)+\nu).
\end{align*}

By putting together these five relations with those appearing in Proposition \ref{Bavula2023Theorem1.3} we get three new and additional relations given by
\begin{equation}\label{newnu1}
    b_1(q_1-1)-ab=0,  \quad  b_2(q_2-1)-\alpha\gamma=0 \quad {\rm and}\quad \mu a=0.
\end{equation}

Similarly, the map $\nu_{x_2}$ can be extended to an algebra homomorphism of $A$ if and only if the equalities
\begin{align*}
     b_1(q_1-1)-ab = &\ 0,\\
     \mu(q_2-1) = &\ \alpha(q_3-1),\\
     b(q_2-1) = &\ \gamma(q_1-1), \\
     b_2(q_1-q_3)+q_1\gamma\mu+b(\mu-q_2\mu-\alpha) = &\ 0, \quad {\rm and} \\
     b_3(q_3-1)-\nu\mu = &\ 0
\end{align*}

are satisfied. These relations together with Proposition \ref{Bavula2023Theorem1.3} and equations {\rm (}\ref{newnu1}{\rm )} generate the new relation
\begin{equation}\label{newnu2}
    b_3(q_3-1)-\nu\mu=0.
\end{equation}

Finally, by considering the extension of the map $\nu_{x_3}$ to an algebra homomorphism of $A$, we get the equations
\begin{align*}
    b_3(q_3-1)-\mu\nu = &\ 0,\\
    \nu(q_1-1) = &\ a(q_3-1), \\ 
    b(q_2-1) = &\ \gamma(q_1-1), \\
    b_1(1-q_2q_3)+\nu\mu(1-q_1)+q_3a\gamma+q_2b\nu = &\ 0,\quad {\rm and} \\ 
    \quad b_2(q_2-1)-\alpha \gamma = &\ 0.
\end{align*}

These relations with all obtained above do not generate a new relation, so we can extend the maps $\nu_{x_1}$, $\nu_{x_2}$ and $\nu_{x_3}$ to an algebra homomorphism if and only if
\begin{align}
    b_1(q_1-1)-ab = &\ 0,\\
    \quad b_2(q_2-1)-\alpha\gamma = &\ 0, \\ 
    b_3(q_3-1)-\mu\nu = &\ 0, \quad\ {\rm and} \\
    \mu a = &\ 0.
\end{align}

Since we need to guarantee that
\begin{align}\label{commuauto}
   \nu_{x_1} \circ \nu_{x_2} = &\ \nu_{x_2} \circ \nu_{x_1}, \\
   \nu_{x_1} \circ \nu_{x_3} = &\  \nu_{x_3} \circ \nu_{x_1}, \quad {\rm and} \\
   \nu_{x_3} \circ \nu_{x_2} = &\ \nu_{x_2} \circ \nu_{x_3}, 
\end{align}

it is enough to satisfy these equalities for the generators $x_1$, $x_2$ and $x_3$, i.e.
\begin{align}
\nu_{x_1} \circ \nu_{x_2}(x_1) = &\ q_1^{-1}x_1-bq_{1}^{-1}, \\
\nu_{x_2} \circ \nu_{x_1}(x_1) = &\ q_1^{-1}x_1-bq_{1}^{-1}, \\ 
\nu_{x_1} \circ \nu_{x_2}(x_2) = &\ q_1x_2+a, \\
\nu_{x_2} \circ \nu_{x_1}(x_2) = &\ q_1x_2+a, \label{comp12} \\
 \nu_{x_1} \circ \nu_{x_2}(x_3) = &\ q_2q_3x_3+q_3\alpha+\mu, \quad {\rm and} \label{comp123} \\
 \nu_{x_2} \circ \nu_{x_1}(x_3) = &\ q_2q_3x_3+q_2\mu+\alpha. \label{comp213}
\end{align}

As we can see, composition $\nu_{x_1} \circ \nu_{x_2} = \nu_{x_2} \circ \nu_{x_1}$ is always satisfied, so we only consider the relations {\rm (}\ref{comp123}{\rm )} and {\rm (}\ref{comp213}{\rm )} which hold when $\alpha(q_3-1)=\mu(q_2-1)$; this is precisely the relation (\ref{Bavula2023(11)}) in Proposition \ref{Bavula2023Theorem1.3}.

Next, 
\begin{align}
    \nu_{x_1} \circ \nu_{x_3}(x_1) = &\ q_2^{-1}x_1-\gamma q_{2}^{-1}, \label{comp13}\\
    \nu_{x_3} \circ \nu_{x_1}(x_1) = &\ q_2^{-1}x_1-\gamma q_{2}^{-1}, \label{comp21} \\
    \nu_{x_1} \circ \nu_{x_3}(x_2) = &\ q_1q_3^{-1}x_2+q_3^{-1}a-\nu q_3^{-1}, \label{comp13'}\\
    \nu_{x_3} \circ \nu_{x_1}(x_2) = &\ q_1q_3^{-1}x_2-\nu q_1q_3^{-1}+a, \label{comp22} \\
    \nu_{x_1} \circ \nu_{x_3}(x_3) = &\ q_2x_3+\alpha, \quad {\rm and} \label{comp22'} \\
    \nu_{x_3} \circ \nu_{x_1}(x_3) = &\ q_2x_3+\alpha. \label{comp23}
\end{align}

Again, the compositions shown in {\rm (}\ref{comp13}{\rm )} and {\rm (}\ref{comp21}{\rm )}, and compositions {\rm (}\ref{comp22'}{\rm )} and {\rm (}\ref{comp23}{\rm )} are the same. Relations {\rm (}\ref{comp13'}{\rm )} and {\rm (}\ref{comp22}{\rm )} coincide when $a (q_3 - 1) = \nu(q_1 - 1)$ which occurs is satisfied due to Proposition \ref{Bavula2023Theorem1.3}.

Finally, 
\begin{align}
    \nu_{x_3} \circ \nu_{x_2}(x_1) = &\ q_1^{-1}q_2^{-1}x_1-q_1^{-1}q_2^{-1}\gamma-bq_1^{-1}, \label{comp31'} \\
    \nu_{x_2} \circ \nu_{x_3}(x_1) = &\ q_1^{-1}q_2^{-1}x_1-q_1^{-1}q_2^{-1}b-\gamma q_2^{-1}, \label{comp31} \\
    \nu_{x_3} \circ \nu_{x_2}(x_2) = &\ q_3^{-1}x_2-\nu q_3^{-1}, \label{comp32'} \\
    \nu_{x_2} \circ \nu_{x_3}(x_2) = &\ q_3^{-1}x_2-\nu q_3^{-1},  \label{comp32} \\
    \nu_{x_3} \circ \nu_{x_2}(x_3) = &\ q_3x_3+\mu, \quad {\rm and} \label{comp33'}\\
    \nu_{x_2} \circ \nu_{x_3}(x_3) = &\ q_3x_3+\mu. \label{comp33}
\end{align}

Relations {\rm (}\ref{comp32'}{\rm )} and {\rm (}\ref{comp32}{\rm )} are equal, and {\rm (}\ref{comp33'}{\rm )} and {\rm (}\ref{comp33}{\rm )} coincide. Note that relations {\rm (}\ref{comp31'}{\rm )} and {\rm (}\ref{comp31}{\rm )} are the same if $\gamma (q_1 - 1) = b (q_2 - 1)$, which is the case by {\rm (}\ref{Bavula2023(13)}{\rm )} in Proposition \ref{Bavula2023Theorem1.3}.

Consider $\Omega^{1}A$ a free right $A$-module of rank three with generators $dx_1$, $dx_2$ and $dx_3$. For all $p\in A$ define a left $A$-module structure by
\begin{align}
    pdx_1 = &\ dx_1 \nu_{x_1}(p), \notag \\ \quad pdx_2 = &\ dx_2\nu_{x_2}(p),\ {\rm and} \notag \\
    pdx_3 = &\ dx_3\nu_{x_3}(p) \label{relrightmod}.
\end{align}

The relations in $\Omega^{1}A$ are given by 
\begin{align}
x_1dx_1 = &\ dx_1 x_1, \notag \\
x_1dx_2 = &\ q_1^{-1}dx_2x_1-bq_1^{-1}dx_2, \notag \\
x_1dx_3 = &\ q_2^{-1}dx_3x_1-\gamma q_2^{-1}dx_3, \label{rel1} \\
x_2dx_1 = &\ q_1dx_1x_2+adx_1, \notag \\  
x_2dx_2 = &\ dx_2x_2, \notag \\
x_2dx_3 = &\ q_3^{-1}dx_3x_2-\nu q_3^{-1}dx_3, \label{rel2} \\
x_3dx_1 = &\ q_2dx_1x_3+\alpha dx_1, \notag \\
x_3dx_2 = &\ q_3dx_2x_3 + \mu dx_2, \ {\rm and} \notag \\
x_3dx_3 = &\ dx_3x_3. \label{rel3} 
\end{align}
    
We want to extend the correspondences 
\begin{equation*}
x_1 \mapsto d x_1, \quad x_2 \mapsto d x_2 \quad {\rm and} \quad x_3\mapsto d x_3
\end{equation*} 

to a map $d: A \to \Omega^{1} A$ satisfying the Leibniz rule. This is possible if it is compatible with the nontrivial relations {\rm (}\ref{Bavula2023(8)}{\rm )}, {\rm (}\ref{Bavula2023(9)}{\rm )} and {\rm (}\ref{Bavula2023(10)}{\rm )}, i.e. if the equalities
\begin{align*}
        dx_2x_1+x_2dx_1 &\ = q_1dx_1x_2+q_1x_1dx_2+adx_1+bdx_2, \\
        dx_3x_1+x_3dx_1 &\ = q_2dx_1x_3+q_2x_1dx_3+\alpha dx_1+\gamma dx_3,\ {\rm and} \\
        dx_3x_2+x_3dx_2 &\ = q_3dx_2x_3+q_3x_2dx_3+\mu dx_2+\nu dx_3
\end{align*}

hold. Note that $d(b_i)=0$ for $i=1,2,3$.

Define $\Bbbk$-linear maps 
\begin{equation*}
\partial_{x_1}, \partial_{x_2}, \partial_{x_3}: A \rightarrow A
\end{equation*}

such that
\begin{align*}
    d(a)=dx_1\partial_{x_1}(a)+dx_2\partial_{x_2}(a)+dx_3\partial_{x_3}(a), \quad {\rm for\ all} \ a \in A.
\end{align*}

Since $dx_1$, $dx_2$ and $dx_3$ are free generators of the right $A$-module $\Omega^1A$, these maps are well-defined. Note that $d(a)=0$ if and only if $\partial_{x_1}(a)=\partial_{x_2}(a)=\partial_{x_3}(a)=0$. By using the three relations in {\rm (}\ref{relrightmod}{\rm )} and the definitions of the maps $\nu_{x_1}$, $\nu_{x_2}$ and $\nu_{x_3}$, we get that
\begin{align}
\partial_{x_1}(x_1^kx_2^lx_3^s) = &\ kx_1^{k-1}x_2^lx_3^s, \notag \\
\partial_{x_2}(x_1^kx_2^lx_3^s) = &\ lq_1^{-k}(x_1-b)^kx_2^{l-1}x_3^s, \quad {\rm and} \notag \\
\partial_{x_3}(x_1^kx_2^lx_3^s) = &\ skq_2^{-k}q_3^{-l}(x_1-\gamma)^k(x_2-\nu)^lx_3^{s-1}.
\end{align}

Thus $d(a)=0$ if and only if $a$ is a scalar multiple of the identity. This shows that $(\Omega A,d)$ is connected where $\Omega A = \Omega^0 A \oplus \Omega^1 A \oplus \Omega^2 A$.

The universal extension of $d$ to higher forms compatible with {\rm (}\ref{rel1}{\rm )}, {\rm (}\ref{rel2}{\rm )} and {\rm (}\ref{rel3}{\rm )} gives the following rules for $\Omega^2A$:
\begin{align}
dx_2\wedge dx_1 = &\ -q_1dx_1\wedge dx_2, \\
dx_3\wedge dx_1 = &\ -q_2dx_1\wedge dx_3, \quad {\rm and} \\ 
dx_3\wedge dx_2 = &\ -q_3dx_2\wedge dx_3 \label{relsecond}.
\end{align}

Since the automorphisms $\nu_{x_1}$, $\nu_{x_2}$ and $\nu_{x_3}$ commute with each other, there are no additional relationships to the previous ones, so we get the expression
\begin{align*}
    \Omega^2A =  dx_1\wedge dx_2A\oplus dx_1\wedge dx_3A\oplus dx_2\wedge dx_3A.
\end{align*}

Since $\Omega^3A = \omega A\cong A$ as a right and left $A$-module, with $\omega=dx_1\wedge dx_2 \wedge dx_3$, where $\nu_{\omega}=\nu_{x_1}\circ\nu_{x_2}\circ\nu_{x_3}$, we have that $\omega$ is a volume form of $A$. From Proposition \ref{BrzezinskiSitarz2017Lemmas2.6and2.7} (2) we get that $\omega$ is an integral form by setting
\begin{align*}
\omega_1^1  = &\ \bar{\omega}_1^1 = dx_1, \\  
\omega_2^1 = &\ \bar{\omega}_2^1 = dx_2, \\
\omega_3^1 = &\ \bar{\omega}_3^1 = dx_3, \\
    \omega_1^2 = &\ dx_2\wedge dx_3, \\
    \omega_2^2 = &\ -q_1^{-1}dx_1\wedge dx_3, \\ 
    \omega_3^2 = &\ q_2^{-1}q_{3}^{-1}dx_1\wedge dx_3, \\
    \bar{\omega}_1^2 = &\ q_1^{-1}q_2^{-1}dx_2\wedge dx_3, \\
    \bar{\omega}_2^2 = &\ -q_3^{-1}dt\wedge dx_2, \quad {\rm and} \\ 
    \bar{\omega}_3^2 = &\ dx_1\wedge dx_2.
\end{align*}

By Proposition \ref{BrzezinskiSitarz2017Lemmas2.6and2.7} (2), we consider the expression $\omega' := dx_1a + dx_2b + dx_3c$ with $a, b, c \in \Bbbk$, to obtain the equality
\begin{align*}
    \sum_{i=1}^{3}\omega_{i}^{1}\pi_{\omega}(\bar{\omega}_i^{2}\wedge \omega') = &\ dx_1\pi_{\omega}(q_1^{-1}q_2^{-1}adx_2\wedge dx_3\wedge dx_1) \\
    & + dx_2\pi_{\omega}(-q_{3}^{-1}bdx_1\wedge dx_3\wedge dx_2) \\
    & + dx_3\pi_{\omega}(cdx_1\wedge dx_2\wedge dx_3) \\
    = &\ dx_1a+dx_2b+dx_3c \\
    = &\ \omega'.
\end{align*}

On the other hand, if $\omega'' := dx_1\wedge dx_2a+dx_1\wedge dx_3 b+dx_2\wedge dx_3 c$ where $a,b,c \in \Bbbk$, we get that
\begin{align*}
\sum_{i=1}^{3}\omega_{i}^{2}\pi_{\omega}(\bar{\omega}_i^{1}\wedge \omega'') = &\  dx_2\wedge dx_3\pi_{\omega}(cdx_1\wedge dx_2 \wedge dx_3) \\
&\ - q_{1}^{-1}dx_1\wedge dx_3\pi_{\omega}(bdx_2\wedge dx_1\wedge dx_3) \\
    &\ + q_{2}^{-1}q_{3}^{-1}dx_1\wedge dx_2\pi_{\omega}(adx_3 \wedge dx_1 \wedge dx_2) \\ 
    = &\ dx_1\wedge dx_2a+dx_1\wedge dx_3b+dx_2\wedge dx_3 \\
    = &\ \omega''.
\end{align*}

As we have seen above, all elements of different degrees can be generated by $\omega_i^j$ and $\bar{\omega}_i^{3-j}$  for $j=1,2$ and $i = 1 , 2, 3$, so Proposition \ref{BrzezinskiSitarz2017Lemmas2.6and2.7} (2) guarantees that $\omega$ is an integral form. Finally, Proposition \ref{integrableequiva} shows that $(\Omega A, d)$ is an integrable differential calculus of degree $n$, and since ${\rm GKdim}(A)$ is also $n$, it follows that $A$ is differentially smooth.
\end{proof}

Following Bavula's classification presented in \cite{Bavula2023}, Theorem \ref{Firsttheoremsmoothnessbi-quadraticalgebras} shows that there are twenty two bi-quadratic algebras on three generators that are differentially smooth. On the other hand, Theorem \ref{Secondtheoremsmoothnessbi-quadraticalgebras} shows that twenty one bi-quadratic algebras cannot be differentially smooth. This is the second important result of the paper.

\begin{theorem}\label{Secondtheoremsmoothnessbi-quadraticalgebras}
If one of the conditions $c\not = 0$, $\beta\not=0$ or $\lambda\not=0$ holds, then $A$ is not differentially smooth.
\end{theorem}
\begin{proof}
By contradiction. Suppose that $A$ has a first order differential calculus $(\Omega A, d)$. Without loss of generality, we consider the case $c\not =0$. Since $d$ must be compatible with the relation {\rm (}\ref{Bavula2023(8)}{\rm )}, then we get that
\begin{equation*}
    dx_2 x_1+x_2dx_1 = q_1dx_1x_2 +q_1x_1dx_2 + adx_1 + bdx_2 + cdx_3,
\end{equation*}

whence $dx_3$ is generated by $dx_1$ and $dx_2$. This means that $\Omega^1A$ is generated by two elements and $\Omega^3A=\Omega^1A\wedge\Omega^1A\wedge\Omega^1A=0$, i.e. there is no third-order calculus. Since ${\rm GKdim}(A) = 3$, we conclude that $A$ cannot be differentially smooth.
\end{proof}

Tables \ref{FirsttableDSBiquadraticalgebras3}, \ref{SecondtableDSBiquadraticalgebras3}, \ref{ThirdtableDSBiquadraticalgebras3}, \ref{FourthtableDSBiquadraticalgebras3}, \ref{FifthtableDSBiquadraticalgebras3}, \ref{SixthtableDSBiquadraticalgebras3}, \ref{SeventhtableDSBiquadraticalgebras3} and \ref{EighttableDSBiquadraticalgebras3} contain the explicit relations of bi-quadratic algebras on three-generators. There, by Theorems \ref{Firsttheoremsmoothnessbi-quadraticalgebras} and \ref{Secondtheoremsmoothnessbi-quadraticalgebras}, the symbols $\checkmark$ and $\star$ denote a positive and a negative answer, respectively, on its differential smoothness.

\begin{remark}
It is a pending task to determine the differential smoothness of the quantum Weyl algebra (Section \ref{DSBiquadraticalgebrasTwogenerators}) and of algebra $D_1$ (Table \ref{SecondtableDSBiquadraticalgebras3}). It is important to note that if one algebra is not differential smooth, this does not mean that the algebra does not possess a differential calculus. The universal enveloping algebra $U(\mathfrak{sl}_2(\Bbbk))$ is an illustration of this fact as it was shown by Beggs and Majid \cite{BeggsMajid2020}. By using Hopf algebras, they proved the existence of a first-order differential calculus and other properties of Riemannian quantum geometry that are different from the differential smoothness shown in this paper.
\end{remark}

\section{Future work}\label{FutureworkDifferentialsmoothnessofbiquadraticalgebras}

As expected, a natural task is to investigate the differential smoothness of bi-quadratic algebras on $n$ generators. Precisely, Bavula \cite[p. 699]{Bavula2023} asserted that a construction of bi-quadratic algebras on four generators was introduced by Zhang and Zhang \cite{ZhangZhang2008, ZhangZhang2009} with their class of {\em double Ore extensions}. As one can appreciate in the literature, these extensions are of great importance in the noncommutative algebraic geometry, and more exactly, in the classification of {\em Artin-Schelter regular algebras} introduced by Artin and Schelter \cite{ArtinSchelter1987} (see the excellent surveys on these algebras carried out by Bellamy et al. \cite{Bellamyetal2016} and Rogalski \cite{Rogalski2023}). The characterization of the differential smoothness of double Ore extensions is one of our immediate objectives.

On the other hand, since bi-quadratic algebras are related to other families of noncommutative rings of polynomial type such as those mentioned in the Introduction, a second natural task is to investigate the extension of the theory developed here to these families of algebras with the aim of studying its differential smoothness.

\newpage

\begin{table}[h]
\caption{Bi-quadratic algebras on three generators of Lie type \cite[Theorem 1.4]{Bavula2023}}
\label{FirsttableDSBiquadraticalgebras3}
\begin{center}
\resizebox{12cm}{!}{
\begin{tabular}{ |c|c|c|c| } 
\hline
Conditions & Bi-quadratic algebra & Matrix form (Proposition \ref{Bavula2023Theorem1.3}) & Differential smoothness \\
\hline
 & $\Bbbk[x_1,x_2,x_3]$ & $\mathbb{A}=0_{3 \times 3}$, $\mathbb{B}=0_{3 \times 1}$ & $\checkmark$ \\ \cline{2-4}
& $U(\mathfrak{sl}_2(\Bbbk))$ & $\mathbb{A} = \begin{bmatrix}
    0 & 0 & -1 \\ 2 & 0 & 0 \\ 0 & -2 & 0
\end{bmatrix}$, $\mathbb{B}=0_{3 \times 1}$ & $\star$ \\ \cline{2-4}
\multirow{6}{*}{$q_1=q_2=q_3=1$} & $U(\mathcal{H}_3)$ & $\mathbb{A} = \begin{bmatrix}
    0 & 0 & -1 \\ 0 & 0 & 0 \\ 0 & 0 & 0
\end{bmatrix}$, $\mathbb{B}=0_{3 \times 1}$ & $\star$ \\ \cline{2-4}
 & $U(\mathcal{N})/(c-1)$ & $\mathbb{A} = \begin{bmatrix}
    0 & 0 & -1 \\ 0 & 0 & 0 \\ 0 & 0 & 0
\end{bmatrix}$, $\mathbb{B} = \begin{bmatrix}
    0  \\ 0  \\ -1
\end{bmatrix}$ & $\star$ \\ \cline{2-4}
& $\Bbbk\langle x_1, x_2, x_3 | [x_1,x_2]=x_2\rangle$ & $\mathbb{A} = \begin{bmatrix}
    0 & -1 & 0 \\ 0 & 0 & 0 \\ 0 & 0 & 0
\end{bmatrix}$, $\mathbb{B}=0_{3 \times 1}$ & $\checkmark$ \\ \cline{2-4}
& $U(\mathcal{M})/(c-1)$ & $\mathbb{A} = \begin{bmatrix}
    0 & -1 & 0 \\ 0 & 0 & 0 \\ 0 & 0 & 0
\end{bmatrix}$, $\mathbb{B} = \begin{bmatrix}
    0  \\ -1  \\ 0
\end{bmatrix}$ & $\checkmark$ \\ \cline{2-4}
\hline
\end{tabular}
}
\end{center}
\end{table}

\begin{table}[h]
\caption{Bi-quadratic algebras on three generators \cite[Theorem 1.5]{Bavula2023}}
\label{SecondtableDSBiquadraticalgebras3}
\begin{center}
\resizebox{12cm}{!}{
\begin{tabular}{ |c|c|c|c| } 
\hline
Conditions & Bi-quadratic algebra & Matrix form (Proposition \ref{Bavula2023Theorem1.3}) & Differential smoothness \\
\hline
 & $A_1$ & $\mathbb{A} = \begin{bmatrix}
    0 & 0 & 0 \\ 1 & 0 & 0 \\ 0 & \mu & 0
\end{bmatrix}$, $\mu\not=-1$, $\mathbb{B}=0_{3 \times 1}$ & $\checkmark$ \\ \cline{2-4}
& $A_2$ & $\mathbb{A} = \begin{bmatrix}
    0 & 0 & 0 \\ 0 & 0 & 0 \\ 0 & 1 & 0
\end{bmatrix}$, $\mathbb{B}=0_{3 \times 1}$ & $\checkmark$ \\ \cline{2-4}
& $B_1$ & $\mathbb{A}=0_{3 \times 3}$, $\mathbb{B}=0_{3 \times 1}$ & $\checkmark$ \\ \cline{2-4}
\multirow{10}{*}{$q_1\not = 1$ and $q_2 = q_3 = 1$} & $B_2$ & $\mathbb{A}=0_{3 \times 3}$, $\mathbb{B} = \begin{bmatrix}
    1  \\ 0  \\ 0
\end{bmatrix}$ & $\checkmark$ \\ \cline{2-4}
& $B_3$ & $\mathbb{A} = \begin{bmatrix}
    0 & 0 & 1 \\ 0 & 0 & 0 \\ 0 & 0 & 0
\end{bmatrix}$, $\mathbb{B}=0_{3 \times 1}$ & $\star$ \\ \cline{2-4}
 & $B_4$ & $\mathbb{A} = \begin{bmatrix}
    0 & 0 & 1 \\ 0 & 0 & 0 \\ 0 & 0 & 0
\end{bmatrix}$, $\mathbb{B} = \begin{bmatrix}
    1  \\ 0  \\ 0
\end{bmatrix}$ & $\star$ \\ \cline{2-4}
& $C_1$ & $\mathbb{A} = \begin{bmatrix}
    0 & 0 & 0 \\ 1 & 0 & 0 \\ 0 & -1 & 0
\end{bmatrix}$, $\mathbb{B}=0_{3 \times 1}$ & $\checkmark$ \\ \cline{2-4}
& $C_2$ & $\mathbb{A} = \begin{bmatrix}
    0 & 0 & 1 \\ 1 & 0 & 0 \\ 0 & -1 & 0
\end{bmatrix}$, $\mathbb{B} = \begin{bmatrix}
    1  \\ 0  \\ 0
\end{bmatrix}$ & $\star$ \\ \cline{2-4}
& $D_1$ & $\mathbb{A} = \begin{bmatrix}
    0 & 0 & 0 \\ 1 & 0 & 0 \\ 0 & -1 & 0
\end{bmatrix}$, $\mathbb{B} = \begin{bmatrix}
    1  \\ 0  \\ 0
\end{bmatrix}$ & ? \\ \cline{2-4}
& $D_2$ & $\mathbb{A} = \begin{bmatrix}
    0 & 0 & 1 \\ 1 & 0 & 0 \\ 0 & -1 & 0
\end{bmatrix}$, $\mathbb{B} = \begin{bmatrix}
    b_1  \\ 0  \\ 0
\end{bmatrix}$, $b_1\not = 1$ & $\star$ \\ 
\hline
\end{tabular}
}
\end{center}
\end{table}

\begin{table}[h]
\caption{Bi-quadratic algebras on three generators \cite[Theorem 1.6]{Bavula2023}}
\label{ThirdtableDSBiquadraticalgebras3}
\begin{center}
\resizebox{12cm}{!}{
\begin{tabular}{ |c|c|c|c| } 
\hline
Conditions & Bi-quadratic algebra & Matrix form (Proposition \ref{Bavula2023Theorem1.3}) & Differential smoothness \\
\hline
&$\mathbb{A}_{(q_1,q_2,1)}^{3}$ & $\mathbb{A}=0_{3 \times 3}$, $\mathbb{B}=0_{3 \times 1}$, $q_1q_2\not=1$ & $\checkmark$\\ \cline{2-4}
 & $E_1$ & $\mathbb{A}=0_{3 \times 3}$, $\mathbb{B}=0_{3 \times 1}$, $q_2=q_1^{-1}$ & $\checkmark$ \\ \cline{2-4}
\multirow{5}{*}{$q_1\not = 1$,\ $q_2\not = 1$ \ and $q_3=1$} & $E_2$ & $\mathbb{A}=0_{3 \times 3}$, $\mathbb{B} = \begin{bmatrix}
    0  \\ 0  \\ 1
\end{bmatrix}$, $q_2=q_1^{-1}$ & $\checkmark$ \\ \cline{2-4}
& $E_3$ & $\mathbb{A} = \begin{bmatrix}
    0 & 0 & 0 \\ 0 & 0 & 0 \\ 1 & 0 & 0
\end{bmatrix}$, $\mathbb{B}=0_{3 \times 1}$, $q_2=q_1^{-1}$ & $\star$ \\ \cline{2-4}
& $E_4$ & $\mathbb{A} = \begin{bmatrix}
    0 & 0 & 0 \\ 0 & 0 & 0 \\ 1 & 0 & 0
\end{bmatrix}$, $\mathbb{B} = \begin{bmatrix}
    0  \\ 0  \\ 1
\end{bmatrix}$, $q_2=q_1^{-1}$ & $\star$ \\ 
\hline
\end{tabular}
}
\end{center}
\end{table}

\newpage

\begin{table}[h]
\caption{Quantum bi-quadratic algebras on three generators \cite[Theorem 1.7]{Bavula2023}}
\label{FourthtableDSBiquadraticalgebras3}
\begin{center}
\resizebox{12cm}{!}{
\begin{tabular}{ |c|c|c|c| } 
\hline
Conditions & Bi-quadratic algebra & Matrix form (Proposition \ref{Bavula2023Theorem1.3}) & Differential smoothness \\
\hline
 & $F_1$ & $\mathbb{A} = \begin{bmatrix}
    0 & 0 & 1 \\ 0 & 1 & 0 \\ 1 & 0 & 0
\end{bmatrix}$, $\mathbb{B} = \begin{bmatrix}
    b_1  \\ b_2  \\ b_3
\end{bmatrix}$, $b_1,b_2,b_3 \in \Bbbk$ & $\star$ \\ \cline{2-4}
& $F_2$ & $\mathbb{A} = \begin{bmatrix}
    0 & 0 & 1 \\ 0 & 1 & 0 \\ 0 & 0 & 0
\end{bmatrix}$, $\mathbb{B} = \begin{bmatrix}
    b_1  \\ b_2  \\ 0
\end{bmatrix}$, $b_1,b_2 \in \Bbbk$ & $\star$ \\ \cline{2-4}
 & $F_3$ & $\mathbb{A} = \begin{bmatrix}
    0 & 0 & 1 \\ 0 & 1 & 0 \\ 0 & 0 & 0
\end{bmatrix}$, $\mathbb{B} = \begin{bmatrix}
    b_1  \\ b_2  \\ 1
\end{bmatrix}$, $b_1,b_2 \in \Bbbk$ & $\star$ \\ \cline{2-4}
\multirow{10}{*}{$q_1 - q_3 = 0$ \ and $1-q_1q_2 = 0$} & $F_4$ & $\mathbb{A} = \begin{bmatrix}
    0 & 0 & 1 \\ 0 & 0 & 0 \\ 0 & 0 & 0
\end{bmatrix}$, $\mathbb{B} = \begin{bmatrix}
    b_1  \\ 0  \\ 0
\end{bmatrix}$, $b_1 \in \Bbbk$ & $\star$ \\ \cline{2-4}
& $F_5$ & $\mathbb{A} = \begin{bmatrix}
    0 & 0 & 1 \\ 0 & 0 & 0 \\ 0 & 0 & 0
\end{bmatrix}$, $\mathbb{B} = \begin{bmatrix}
    b_1  \\ 1  \\ 0
\end{bmatrix}$, $b_1 \in \Bbbk$ & $\star$ \\ \cline{2-4}
& $F_6$ & $\mathbb{A} = \begin{bmatrix}
    0 & 0 & 1 \\ 0 & 0 & 0 \\ 0 & 0 & 0
\end{bmatrix}$, $\mathbb{B} = \begin{bmatrix}
    b_1  \\ 1  \\ 1
\end{bmatrix}$, $b_1 \in \Bbbk$ & $\star$ \\ \cline{2-4}
& $F_7$ & $\mathbb{A}=0_{3 \times 3}$, $\mathbb{B}=0_{3 \times 1}$ & $\checkmark$ \\ \cline{2-4}
& $F_8$ & $\mathbb{A}=0_{3 \times 3}$, $\mathbb{B} = \begin{bmatrix}
    1  \\ 0  \\ 0
\end{bmatrix}$ & $\checkmark$ \\ \cline{2-4}
& $F_9$ & $\mathbb{A}=0_{3 \times 3}$, $\mathbb{B} = \begin{bmatrix}
    1  \\ 1  \\ 0
\end{bmatrix}$ & $\checkmark$ \\ \cline{2-4}
& $F_{10}$ & $\mathbb{A}=0_{3 \times 3}$, $\mathbb{B} = \begin{bmatrix}
    1  \\ 1  \\ 1
\end{bmatrix}$ & $\checkmark$ \\ 
\hline
\end{tabular}
}
\end{center}
\end{table}

\begin{table}[h]
\caption{Bi-quadratic algebras on three generators \cite[Theorem 1.8]{Bavula2023}}
\label{FifthtableDSBiquadraticalgebras3}
\begin{center}
\resizebox{12cm}{!}{
\begin{tabular}{ |c|c|c|c| } 
\hline
Conditions & Bi-quadratic algebra & Matrix form (Proposition \ref{Bavula2023Theorem1.3}) & Differential smoothness \\
\hline
 & $G_1$ & $\mathbb{A}=0_{3 \times 3}$, $\mathbb{B}=0_{3 \times 1}$ & $\checkmark$ \\ \cline{2-4}
\multirow{4}{*}{$q_1 - q_3 = 0$\ and $1 - q_1q_2 \not = 0$} & $G_2$ & $\mathbb{A}=0_{3 \times 3}$, $\mathbb{B} = \begin{bmatrix}
    0  \\ 1  \\ 0
\end{bmatrix}$ & $\checkmark$ \\ \cline{2-4}
& $G_3$ & $\mathbb{A} = \begin{bmatrix}
    0 & 0 & 0 \\ 0 & 1 & 0 \\ 0 & 0 & 0
\end{bmatrix}$, $\mathbb{B}=0_{3 \times 1}$ & $\star$ \\ \cline{2-4}
& $G_4$ & $\mathbb{A} = \begin{bmatrix}
    0 & 0 & 0 \\ 0 & 1 & 0 \\ 0 & 0 & 0
\end{bmatrix}$, $\mathbb{B} = \begin{bmatrix}
    0  \\ 1  \\ 0
\end{bmatrix}$ &  $\star$ \\ 
\hline
\end{tabular}
}
\end{center}
\end{table}

\begin{table}[h]
\caption{Bi-quadratic algebras on three generators \cite[Theorem 1.9]{Bavula2023}}
\label{SixthtableDSBiquadraticalgebras3}
\begin{center}
\resizebox{12cm}{!}{
\begin{tabular}{ |c|c|c|c| } 
\hline
Conditions & Bi-quadratic algebra & Matrix form (Proposition \ref{Bavula2023Theorem1.3}) & Differential smoothness \\
\hline 
 & $H_1$ & $\mathbb{A}=0_{3 \times 3}$, $\mathbb{B}=0_{3 \times 1}$ & $\checkmark$ \\ \cline{2-4}
\multirow{4}{*}{$q_1 - q_3 \neq 0$,\ $1-q_1q_2=0$\ and $1-q_2q_3\neq 0$ } & $H_2$ & $\mathbb{A}=0_{3 \times 3}$, $\mathbb{B} = \begin{bmatrix}
    0  \\ 0  \\ 1
\end{bmatrix}$ & $\checkmark$ \\ \cline{2-4}
& $H_3$ & $\mathbb{A} = \begin{bmatrix}
    0 & 0 & 0 \\ 0 & 0 & 0 \\ 1 & 0 & 0
\end{bmatrix}$, $\mathbb{B}=0_{3 \times 1}$ & $\star$ \\ \cline{2-4}
& $H_4$ & $\mathbb{A} = \begin{bmatrix}
    0 & 0 & 0 \\ 0 & 0 & 0 \\ 1 & 0 & 0
\end{bmatrix}$, $\mathbb{B} = \begin{bmatrix}
    0  \\ 0  \\ 1
\end{bmatrix}$ & $\star$ \\ \cline{2-4}
\hline
\end{tabular}
}
\end{center}
\end{table}

\begin{table}[h]
\caption{Bi-quadratic algebras on three generators \cite[Theorem 1.10]{Bavula2023}}
\label{SeventhtableDSBiquadraticalgebras3}
\begin{center}
\resizebox{12cm}{!}{
\begin{tabular}{ |c|c|c|c| } 
\hline
Conditions & Bi-quadratic algebra & Matrix form (Proposition \ref{Bavula2023Theorem1.3}) & Differential smoothness \\
\hline
 & $I_1$ & $\mathbb{A}=0_{3 \times 3}$, $\mathbb{B}=0_{3 \times 1}$ & $\checkmark$\\ \cline{2-4}
\multirow{4}{*}{$q_1 - q_3 \neq 0$,\ $1-q_1q_2 \neq 0$\ and $1 - q_2q_3 = 0$ } & $I_2$ & $\mathbb{A}=0_{3 \times 3}$, $\mathbb{B} = \begin{bmatrix}
    1  \\ 0  \\ 0
\end{bmatrix}$ & $\checkmark$ \\ \cline{2-4}
& $I_3$ & $\mathbb{A} = \begin{bmatrix}
    0 & 0 & 1 \\ 0 & 0 & 0 \\ 0 & 0 & 0
\end{bmatrix}$, $\mathbb{B}=0_{3 \times 1}$ & $\star$ \\ \cline{2-4}
& $I_4$ & $\mathbb{A} = \begin{bmatrix}
    0 & 0 & 1 \\ 0 & 0 & 0 \\ 0 & 0 & 0
\end{bmatrix}$, $\mathbb{B} = \begin{bmatrix}
    1  \\ 0  \\ 0
\end{bmatrix}$ & $\star$ \\
\hline
\end{tabular}
}
\end{center}
\end{table}

\begin{table}[h]
\caption{Bi-quadratic algebras on three generators \cite[Theorem 1.11]{Bavula2023}}
\label{EighttableDSBiquadraticalgebras3}
\begin{center}
\resizebox{12cm}{!}{
\begin{tabular}{ |c|c|c|c| } 
\hline
Conditions & Bi-quadratic algebra & Matrix form (Proposition \ref{Bavula2023Theorem1.3}) & Differential smoothness \\
\hline
\multirow{3}{*}{$q_1 - q_3 \neq 0$, \ $1 - q_1q_2 \neq 0$\ and $1 - q_2q_3 \neq 0$} &  &  & \\
& $\mathbb{A}_{Q}^{3}$ & $\mathbb{A}=0_{3 \times 3}$, $\mathbb{B}=0_{3 \times 1}$ & $\checkmark$\\
& &  & \\ 
\hline
\end{tabular}
}
\end{center}
\end{table}

\end{document}